\documentclass{article}

\usepackage{arxiv}

\usepackage[utf8]{inputenc} % allow utf-8 input
\usepackage[T1]{fontenc}    % use 8-bit T1 fonts
\usepackage{hyperref}       % hyperlinks
\usepackage{url}            % simple URL typesetting
\usepackage{booktabs}       % professional-quality tables
\usepackage{amsfonts}       % blackboard math symbols
\usepackage{nicefrac}       % compact symbols for 1/2, etc.
\usepackage{microtype}      % microtypography
\usepackage{lipsum}
\usepackage{amsmath}

\usepackage{amsfonts,amssymb,stmaryrd}
\usepackage{graphicx}
\usepackage{amsthm}
\usepackage{xcolor}
\usepackage{soul}
\usepackage{paralist} 

\newcommand{\R}{\mathbb{R}}

\newtheorem{theorem}{Theorem}[section]
\newtheorem{proposition}{Proposition}[section]
\newtheorem{lemma}{Lemma}[section]

\newtheorem{remark}{Remark}[section]

\usepackage{subcaption}
\usepackage{cite}
\usepackage{authblk}
\newcommand{\p}{\partial}
\newcommand{\bb}{\begin{equation}}
\newcommand{\ee}{\end{equation}}
\newcommand{\ba}{\begin{array}}
\newcommand{\ea}{\end{array}}
\newcommand{\f}{\frac}
\usepackage[all]{xy}
\newcommand{\ds}{\displaystyle}

\newcommand{\al}{\alpha}

\newcommand{\s}{\mathbb{S}}

\newcommand{\sign}{\text{sgn}\,}

\newcommand{\N}{{\mathbb N}}

\usepackage{ stmaryrd }

\numberwithin{equation}{section}

\usepackage{ulem}

\usepackage[most]{tcolorbox}
\definecolor{TextColor}{rgb}{0.75, 0.75, 0.75}

\title{Global and blow-up solutions for a non-local integrable equation with applications to geometry}

\author[1]{Nilay Duruk Mutlubas}  \author[2]{Igor Leite Freire}
\affil[1]{Faculty of Engineering and Natural Sciences, Sabanci University, Turkey
\texttt{nilay.duruk@sabanciuniv.edu}}

\affil[2]{ Universidade Federal de São Carlos,
São Carlos-SP, Brasil\\
  \texttt{igor.leite.freire@gmail.br}}
 
\begin{document}
\maketitle
\begin{abstract}
We establish the global existence of higher-order Sobolev solutions for a non-local integrable evolution equation arising in the study of pseudospherical surfaces and non-linear wave propagation. Under a natural assumption on the initial momentum, we prove that the solution remains globally regular in arbitrary finite-order Sobolev spaces. The proof relies on an inductive energy method involving a hierarchy of functional estimates and applies to both the periodic and non-periodic settings. We determine a criterion for the existence of blow-up solutions. The consequences of these qualitative properties of the solutions on Riemannian surfaces determined by the solutions of the equation are investigated.
\end{abstract}


{\bf MSC classification 2020:} 35B45, 37K10, 53B20.

\section{Introduction}

The study of integrable equations has been an active subject of research since the seminal work by Zabusky and Kruskal \cite{zab}, when the soliton solutions of the KdV equation were first reported.

Since then a large number of researchers have been driven to these equations. This is readily understood, since this sort of equation has rich structural properties that can be investigated from many different standpoints, ranging from algebraic geometry to analysis of PDEs. The reader is guided to the book by Kasman \cite{kasman}, where some of these aspects are discussed.

Around 30 years ago, Camassa and Holm \cite{chprl} rediscovered an integrable equation having peaked solutions, that is, solutions that, far from a certain line in space-time, behave like a smooth decaying travelling waves away from the crest, but their derivatives are discontinuous along the peak. These solutions are known as {\it peakons}. 

The Camassa-Holm (CH) equation has been proved to be a seminal model and it attracted, and has still been attracting, considerable interest. This explains why many studies have focused on Camassa-Holm type equations.

Some years ago, Novikov carried out an extensive classification of CH type equations \cite{nov} of the form
$$
u_t-u_{txx}=F(u,u_x,u_{xx},u_{xxx}).
$$

By virtue of the presence of the Helmholtz operator $\Lambda^{2}=1-\p_x^2$, these equations can be rewritten in a non-local, evolution form. In this paper we focus on qualitative properties of the following equation:\bb\label{1.1}
u_t-u_{txx}=\p_x(2-\p_x)(1+\p_x)u^2.
\ee

Equations of this type arise in modelling dispersive shallow water waves and integrable systems \cite{chprl,nov}, where non-local effects and non-linear interactions are fundamental in the description of wave behaviour. 

In regard to equation \eqref{1.1}, it has been studied in recent years; see, for instance, \cite{li-na, li-jmaa, liu-jde, pri-na,nilay, guo} where qualitative properties of \eqref{1.1} are studied. More recently, it was shown in \cite{tito,freire-tito-sam,nazime} that \eqref{1.1} is geometrically relevant in the sense that its solutions determine metrics for surfaces of Gaussian curvature ${\cal K}=-1$.

In \cite{nilay} we considered how periodic solutions affect the geometry of the corresponding surface. To this end, in \cite{nilay} we studied \eqref{1.1} simultaneously from the point of view of Cauchy problems, leading to an analysis of PDEs problems and qualitative properties of the metrics determined by these solutions.

Equation \eqref{1.1} can be transformed into the following non-local evolution equation
\bb\label{1.2}
u_t-2uu_x=\p_x\Lambda^{-2}(u^2+(u^2)_x),
\ee
which is more convenient from the point of view of qualitative analysis.

Henceforth $\mathbb{K}$ denotes either $\R$, when non-periodic problems come into play, or $\s=[0,1)$, for the periodic case. In \eqref{1.2}, $\Lambda^{-2}f:=g\ast f$, $\p_x\Lambda^{-2}f=(\p_xg)\ast f$, $\ast$ denotes the usual convolution, and
\bb\label{1.3}
\ba{lll}
    \ds{g(x)=\f{1}{2}e^{-|x|}}, &\ds{\p_x g(x)=-\f{\sign{(x)}}{2}e^{-|x|}},&\,\,\text{when }\mathbb{K}=\R,\\
    \\
    \ds{g(x)=\f{\cosh{(x-\lfloor x \rfloor-1/2)}}{2\sinh{(1/2)}}}, &\ds{\p_x g(x)=-\f{\sinh{(x-\lfloor x \rfloor-1/2)}}{2\sinh{(1/2)}}},&\,\,\text{when }\mathbb{K}=\s,
\ea
\ee
where $\lfloor \cdot \rfloor$ denotes the greatest integer function.

\begin{remark}\label{rem1.1}
    In this paper $H^s(\mathbb{K})$ denotes the usual Sobolev space of order $s\in\R$. The norms of a function $u$ belonging to a Banach space $Y$ will be denoted by $\|u\|_Y$. Usually the functions we will deal with have time dependence. That being, the norm of a function $u(t,\cdot)\in Y$ will be denoted by $\|u(t,\cdot)\|_Y$.
\end{remark}

Strictly speaking, \eqref{1.2} and \eqref{1.1} are not equivalent equations. A discussion for this fact for the CH equation can be found in \cite{freire-ch}. Mutatis mutandis, the same discussion is applied to \eqref{1.1}-\eqref{1.2}. Depending on the function space in which the solutions lie in, they can however be considered as equivalent. In particular, this is true for Sobolev spaces $H^s(\mathbb{K})$, for suitable $s$. 

From the point of view of analysis, a strong solution to \eqref{1.2} is a $C^1$ function which solves the equation, whereas a strong solution to \eqref{1.1} must be $C^3$ with respect to $x$ and $C^1$ considering $t$. Due to geometric reasons, the solutions we had to consider in \cite{nilay} were solutions for both equations. In particular, our objects were members of the class 
$
C^0([0,T),H^4(\mathbb{K}))\cap C^1([0,T);H^3(\mathbb{K}))
$
where the lifespan $T$ is a value determined by the initial datum. If one compares this function space with those typically considered in the analysis of PDEs, such as \cite{li-na, li-jmaa, liu-jde}, we see that the results proved in \cite{nilay} were concerned with more regular functions than those typically studied in analysis. The reason for that is: a solution has to be at least $C^3$ (in $x$) in order for it to define a metric for a surface. A similar problem appears when geometry and analysis are conjunctively studied for this sort of equations, see also \cite{pri-jde, freire-ch, nilay, freire-dp}.

In \cite{freire-AML}, solutions of higher regularity for the CH and the Degasperis-Procesi equations were considered. Due to this work, as well as our prior results dealing with geometric analysis of \eqref{1.1}, in this paper we consider the problem of finding solutions of \eqref{1.1} when the initial momentum $m_0:=u_0-u''_0\in H^n(\mathbb{K})$, where $n\in\N:=\{1,2,3,\cdots\}$.

Our main result concerning the existence of global solutions of \eqref{1.1} is:

\begin{theorem}\label{teo1.1}
Let $n\in\N$ and $m_0:=u_0-u_0''\in H^n(\mathbb{K})\cap L^1(\mathbb{K})$. If $m_0(x)\geq0$, for any $x\in \mathbb{K}$, then the corresponding solution $u$ of 
\bb\label{1.4}
\left\{
\ba{lcl}
u_t-u_{txx}=\p_x(2-\p_x)(1+\p_x)u^2,\\
\\
u(0,x)=u_0(x),
\ea
\right.
\ee
$x\in \mathbb{R}$, exists globally in time, that is, $u\in C^{0}([0,\infty),H^{n+2}(\mathbb{K}))\cap C^{1}([0,\infty),H^{n+1}(\mathbb. {K}))$. Whenever $\mathbb{K}=\s$ we have the additional periodic condition $u(t,x)=u(t,x+1)$.
\end{theorem}

\begin{remark}\label{rem1.2}
    We do not address the problem of local existence of solutions of \eqref{1.4} in Sobolev spaces $H^s(\mathbb{K})$ in the present work since it is a topic already previously investigated. For the periodic case the reader can find it in \cite[Theorem 1.2]{nilay} while \cite[Theorem 3.2]{li-na} tackle the non-periodic case. Therefore, from now on we shall assume without further mention that as long as $u_0\in H^s(\mathbb{K})$, $s>3/2$, then \eqref{1.4} has a unique local solution $u\in C^0([0,T), H^s(\mathbb{K}))\cap C^1([0,T), H^{s-1}(\mathbb{K}))$, for some $T$ depending on $u_0$.
\end{remark}

Once the problem of global existence is addressed, a natural question is to clarify whether blow-up solutions may arise. Regarding the periodic problem, to the best of our knowledge, no work has answered this point, whereas for the non-periodic case we have some blow-up results reported in \cite[Theorem 5.1]{li-na}, \cite[Proposition 2]{guo} and \cite[Theorem 4.2]{li-jmaa}.

We now present a blow-up result that holds for both periodic and non-periodic case. For the latter, its conditions on the initial datum is different of those used in the aforementioned references.

\begin{theorem}\label{teo1.2}
    Assume that $u_0\in H^4(\mathbb{K})$, $(u_0-u_0')(x)<0$, for all $x\in \mathbb{K}$, and $u_0-u_0'\in L^1(\mathbb{K})$; $\sigma_0\in\Big(\|G\|_{L^\infty(\mathbb{K})}\|u_0-u_0'\|_{L^1(\mathbb{K})},\|G\|_{L^\infty(\mathbb{K})}\|u_0-u_0'\|_{L^1(\mathbb{K})}+\f{1}{2}\Big)$. If there exists a point $x_0\in\mathbb{K}$ such that $(u_0-u_0')(x_0)=-\sigma_0$, then the corresponding solution to \eqref{1.1} blows-up at
    $$T_0=-\f{1}{\|G\|_{L^\infty(\mathbb{K})}\|u_0-u_0'\|_{L^1(\mathbb{K})}}\ln{\Big(1-\f{\|G\|_{L^\infty(\mathbb{K})}\|u_0-u_0'\|_{L^1(\mathbb{K})}}{\sigma_0}\Big)}.$$
\end{theorem}

Above $G$ denotes the Green function of the operator $(1-\p_x)$, which is given by
\bb\label{1.5}
G(x)=\left\{\ba{ll}
    \ds{e^{-x}H(x)}, &\,\,\text{when }\mathbb{K}=\R,\\
    \\
    \ds{\f{e^{x-\lfloor x \rfloor}}{e-1}},&\,\,\text{when }\mathbb{K}=\s,
\ea
\right.
\ee
where $H(\cdot)$ denotes the Heaviside step function.

Theorems \ref{teo1.1} and \ref{teo1.2} are key ingredients for proving the next geometric results.

\begin{theorem}\label{teo1.3} 
Let $m_0=u_0-u_0''\in H^2(\mathbb{K})\cap L^1(\mathbb{K})$ be an initial datum with $u_0-u_0''>0$, $u$ be the corresponding solution to \eqref{1.4} and
\bb\label{1.6}
\ba{lcl}
\omega_1&=&\Big(u-u_{xx}\Big)dx+\Big(2u(u-u_{xx})-2(u-u_x)^2\Big)dt,\\
\\
\omega_2&=&\Big(\mu (u-u_{xx})\pm \sqrt{1+\mu^{2}}\Big)dx+\ds{\mu\big(2u(u-u_{xx})-2(u-u_x)^2\big)}dt,\\
\\
\omega_3&=&\Big(\pm\sqrt{1+\mu^{2}}(u-u_{xx})+\mu \Big)dx\pm\Big(\ds{\sqrt{1+\mu^{2}}\big(2u(u-u_{xx})-2(u-u_x)^2\big)}\Big)dt.
\ea
\ee

Then the one-forms \eqref{1.6} are defined on $U:=\R\times(0,\infty)$ and for any open, connected set of $U$ such that $\nabla u\neq(0,0)$ is ensured a pseudospherical surface (PSS) structure.
\end{theorem}

In this work we do not pay detailed attention to PSS's and PDE's. We guide the reader to \cite{chern, nilay,reyesjmp2000, reyes2000, reyes2002, reyes2006-sel, reyes2006-jde, reyes2011,pri-jde,freire-tito-sam,freire-ch,freire-dp,nazime,tito} for further details. In \cite{nilay,pri-jde,freire-ch,freire-dp,nazime} there are deeper discussions concerning PSS's determined by the solutions of Cauchy problems.

The one-forms \eqref{1.6} satisfy 
$$
\ba{lcl}
d\omega_1-\omega_3\wedge\omega_2&=&{\cal E}dx\wedge dt,\quad
d\omega_2-\omega_1\wedge\omega_3={\cal E}dx\wedge dt,\\
\\
d\omega_3-\omega_1\wedge\omega_2&=&\pm\sqrt{1+\mu^2}{\cal E}dx\wedge dt,
\ea
$$
where ${\cal E}=u_{t}-u_{txx}-4uu_{x}-2u_{x}^{2}-2uu_{xx}+6u_{x}u_{xx}+2uu_{xxx}$, see \cite[Equation (2.16)]{nilay}, meaning that
\bb\label{1.7}
\left\{
\ba{lcl}
d\omega_1&=&\omega_3\wedge\omega_2,\\
\\
d\omega_2&=&\omega_1\wedge\omega_3,\\
\\
d\omega_3&=&\omega_1\wedge\omega_2
\ea
\right.
\ee
if and only if $u$ in \eqref{1.6} is a solution of \eqref{1.1}.

For those solutions $u$ for which $\omega_1\wedge\omega_2\neq0$, then $u$ defines a pseudospherical surface (that is, a surface of Gaussian curvature ${\cal K}=-1$), in the sense that \eqref{1.6} satisfies the structure equations for such a surface. For further details, see \cite{chern, reyesjmp2000, reyes2000, reyes2002, reyes2006-sel, reyes2006-jde, reyes2011}. For a more in-depth discussion in connection with Cauchy problems, see \cite{freire-ch,freire-dp,freire2026}. Explicitly, the metric (first fundamental form) $g=\omega_1^2+\omega_2^2$ for the surface determined by \eqref{1.6} is
\bb\label{1.8}
\ba{lcl}
g&=&\ds{\Big[(u-u_{xx})^2+\Big(\mu (u-u_{xx})\pm \sqrt{1+\mu^2}\Big)^2\Big]dx^2}\\
\\
&+&2\ds{\Big(2u(u-u_{xx})-2(u-u_x)^2\Big)\Big[(1+\mu^2)(u-u_{xx}) \pm \mu\sqrt{1+\mu^2}\Big]dxdt}\\
\\
&+&\ds{(1+\mu^2)\Big(2u(u-u_{xx})-2(u-u_x)^2\Big)^2dt^2}.
\ea
\ee

\begin{theorem}\label{teo1.4}
    Let $u_0\in H^4(\mathbb{K})$ be an initial datum satisfying the condition in Theorem \ref{teo1.2} and $u$ be the corresponding solution to \eqref{1.4}. Then \eqref{1.8} defines the metric of a surface with Gaussian curvature ${\cal K}=-1$ such that $g_{12}\rightarrow\infty$ as $t$ approaches $T_0$ given in Theorem \ref{teo1.2}.
\end{theorem}

{\bf Novelty of the manuscript.} Our work addresses the problem of higher order regular solutions with respect to the $x-$variable. This has an evident relevance in the study of surfaces, like those performed in \cite{pri-na, pri-jde,freire-tito-sam,freire-ch,nilay,nazime,tito}. Usually works studying qualitative properties of solutions have a framework of lower regularity, although higher regularity is not an unheard topic in qualitative studies, see \cite{pri-na,liu-jde}. However, in line with these two references, when higher regular solutions are considered, usually they are investigated in the analytic level. In general, when finite regularity is investigated, it rarely goes up to arbitrary, but finite order, with respect to $x$. This paper intends to fill this gap. 

One point we would like to emphasise here is the approach we followed. We divided the proof of {Theorem \ref{1.1}} into several technical steps, so that the reader can follow step by step an extremely technical demonstration. Moreover, the way we present our result can probably be adapted for other similar equations when addressing analogous problems.

Our next target is the investigation of the existence of finite time solutions. For the non-periodic case some previous works, such as \cite{li-na,li-jmaa} have already established conditions on the initial datum so that solutions of \eqref{1.4} blow up. However, to the best of our knowledge, nothing has been reported for periodic solutions. Our Theorem \ref{teo1.2} not only presents a scenario for which we have periodic solutions blowing-up, but also gives new conditions for the existence of finite solutions for the non-periodic case.

Finally, we explore the consequences of the solutions described by our theorems \ref{teo1.1} and \ref{1.2} on the PSS determined by \eqref{1.6}--\eqref{1.8} determined by the solutions of \eqref{1.4} with initial data satisfying either the conditions in Theorem \ref{teo1.1} or \ref{teo1.2}.

{\bf Outline and strategy of the paper.} In Section \ref{sec2}, we present a number of propositions enclosing nearly all technicalities we have to tackle in order to prove our main result. We close it with a theorem that, essentially, gives sufficient conditions for the existence of a global solution at the $H^3(\mathbb{K})$ level. Once such a result is established, we can reach $H^{n+2}(\mathbb{K})$ level, for any $n\geq 1$. In Section \ref{sec3}, we give a proof of {Theorem \ref{1.1}}.

Our strategy to establish the main result in this paper is the following:
\begin{itemize}
\item[(a)] We define a hierarchy of functionals $I_n(t)$, each one equivalent to $\|u(t,\cdot)\|_{H^n}-$norm; 
\item[(b)] We then prove that $I_n(t)$ is bounded provided that the member $I_{n-1}(t)$ is bounded;
\item[(c)] We show that the conditions in Theorem \ref{teo1.1} implies 
$$\|u\|_{H^1}^2\leq \|u_0\|_{H^1}^2 e^{10\int_0^t\|u_x\|_{L^\infty(\mathbb{K})}d\tau}\leq  \|u_0\|_{H^1}^2 e^{At}=:C_0^2(t),$$
for some optimal constant $A>0$ since $\|u_x\|_{L^\infty(\mathbb{K})}$ norm is bounded. This estimate is valid at any finite time, therefore is a global bound for $H^1(\mathbb{K})$ norm of $u$. We provide this inequality to show that  $\|u\|_{L^\infty(\mathbb{K})}$ is also bounded, since
$\|u\|_{L^\infty(\mathbb{K})} \leq  \|u\|_{H^1} \leq C_0$ by Sobolev embedding theorem. 
\item[(d)] The two facts above trigger a process that ensures the validity of our main result.
\end{itemize}

Concerning the blow-up of solutions, unlike other CH type equations, where the operator $1-\p_x^2$ and its inverse play vital importance for addressing problems of qualitative nature, e.g, see \cite{const1997, const1998-1,const1998-2,const1998-3,const2000-1}, for \eqref{1.1} we instead use the operator $1-\p_x$ and its inverse. This allows us to establish conditions for $u-u_x\rightarrow-\infty$ as long as $t$ approaches a finite value determined by the $L^1(\mathbb{K})$ norm of the quantity $u-u_0'$. This is done in section \ref{sec4}. 

The geometric consequences of theorems \ref{teo1.1} and \ref{teo1.2}, namely, theorems \ref{teo1.3} and \ref{teo1.4}, are established in section \ref{sec5}.

\section{Preliminaries}\label{sec2}

Henceforth, unless otherwise stated, $n\in\N$ is a fixed natural number. Throughout this section we assume that $u$ is $C^n$ with respect to $x$ and $C^1$ with respect to $t$ and a member of $L^1(\mathbb{K})$ as well. Also, it is presupposed that $(1+\p_x)\Lambda^{-2}u^2$ is formally well-defined. In addition, $c_n$ stands for a generic constant depending on $n$

From the Leibniz formula for differentiation, we have
\begin{equation}\label{2.1}
    \p_x^n u^2=\sum_{k=0}^n \binom{n}{k} (\p_x^{n-k}u)(\p_x^{k}u),
\end{equation}
where
$$
    \binom{n}{k} = \left\{
    \ba{lcl}
    \ds{\frac{n!}{k!(n-k)!}}&\text{if}&n\geq k,\\
    \\
    0&\text{otherwise},&
    \ea\right.
$$
is the usual binomial coefficient. We define the quantities
\bb\label{2.2}
k_{n}(t):=\max_{0\leq j\leq n}\|\p_x^j u(t,\cdot)\|_{L^\infty(\mathbb{K})},
\ee
\bb\label{2.3}
J_m(t,x):=\f{1}{2}\sum_{j=0}^m \p_x^ju(t,x)^2,\quad 0\leq m\leq n
\ee
and
\bb\label{2.4}
I_m(t)=\int_\mathbb{K} J_m(t,x)dx,\quad 0\leq m\leq n.
\ee

\begin{proposition}\label{prop3.1}
    For $0\leq k\leq \ell<n$, we have
    $
    |\p_x^k u||\p_x^\ell u||\p_x^n u|\leq k_{n-1}(t) J_n(t,x).
    $
\end{proposition}

\begin{proof}
It follows from \eqref{2.2}, \eqref{2.3} and the fact that $2ab\leq a^2+b^2$, for any numbers $a$ and $b$.
\end{proof}

\begin{proposition}\label{prop3.2}
    We have $|\p^n_xu^2|\leq 2^n J_n$.
\end{proposition}

\begin{proof}
    From \eqref{2.1} we have
    $$
    |\p_x^n u^2|\leq \sum\limits_{k=0}^n \binom{n}{k}\f{(\p_x^{n-k}u)^2+(\p_x^{k}u)^2}{2}\leq \sum\limits_{k=0}^n \binom{n}{k} J_n(t,x)=2^n J_n(t,x),
    $$
    which is the desired result.
\end{proof}

\begin{proposition}\label{prop3.3}
    The following identity formally holds:
    $$
    \p_x^n(1+\p_x)\Lambda^{-2}=(1+\p_x)\Lambda^{-2}-\sum\limits_{k=0}^{n-1}\p_x^k,
    $$ where $\p_x^0:=1$.
\end{proposition}

\begin{proof}
    Since $\Lambda^2=1-\p_x^2$, then $\p_x^2\Lambda^{-2}=\Lambda^{-2}-1$. Let us define
    $$\N':=\Big\{n\in\N;\,\,\p_x^n(1+\p_x)\Lambda^{-2}=(1+\p_x)\Lambda^{-2}-\sum\limits_{k=0}^{n-1}\p_x^k\Big\}$$ and prove that $\N'=\N$. In fact: \begin{itemize}
\item[(a)] $1\in\N'$. We note that
$$
\p_x(1+\p_x)\Lambda^{-2}=\p_x\Lambda^{-2}+\p_x^2\Lambda^{-2}=(1+\p_x)\Lambda^{-2}-1,
$$
where we used the identity above.
\item[(b)] Now suppose that $n\in\N'$. Then
$$
\p_x^{n+1}(1+\p_x)\Lambda^{-2}=\p_x\Big((1+\p_x)\Lambda^{-2}-\sum\limits_{k=0}^{n-1}\p_x^k\Big)=(1+\p_x)\Lambda^{-2}-\sum\limits_{k=0}^{n}\p_x^k.
$$
\end{itemize}
As a result, the fact that $n\in\N'$ implies that $n+1$ is a member of the same set as well. Therefore, we must have $\N'=\N$.
\end{proof}

The next result is concerned with integration by parts. 

\begin{proposition}\label{prop3.4}
    Provided that $(\p_x^ku)\p_x^{n-k} (1+\p_x)\Lambda^{-2}u^2\Big|_\mathbb{K}$ vanishes, $0\leq k\leq n$, the following identity holds
    $$
    \Big|\int_\mathbb{K}(\p_x^n u)(1+\p_x)\Lambda^{-2}u^2dx\Big|\leq \int_\mathbb{K}\Big|u(1+\p_x)\Lambda^{-2}u^2\Big|dx+(2^n-1)k_{n-1}(t) I_n(t).
    $$
\end{proposition}

\begin{proof}
Integration by parts combined with Proposition \ref{prop3.3} read
$$
\ba{l}
\ds{\int_\mathbb{K}(\p_x^n u)(1+\p_x)\Lambda^{-2}u^2dx}=\ds{(-1)^n\int_\mathbb{K} u(\p_x^n(1+\p_x)\Lambda^{-2}u^2)dx}\\
\\
=\ds{(-1)^n\int_\mathbb{K} u\Big[(1+\p_x)\Lambda^{-2}u^2-\sum_{k=0}^{n-1}\p_x^ku^2\Big]dx}
=\ds{(-1)^n\int_\mathbb{K} u(1+\p_x)\Lambda^{-2}u^2dx+(-1)^{n+1}\sum_{k=0}^{n-1}\int_\mathbb{K}u\p_x^ku^2dx}.
\ea
$$

Using \eqref{2.1} once more, we have
$$
\ba{lcl}
\ds{\Big|\int_\mathbb{K}(\p_x^n u)(1+\p_x)\Lambda^{-2}u^2dx\Big|}&\leq&\ds{\Big|\int_\mathbb{K} u(1+\p_x)\Lambda^{-2}u^2dx\Big|+\sum_{k=0}^{n-1}\sum_{j=0}^k\binom{k}{j}\int_\mathbb{K}|u||\p_x^{k-j}u||\p_x^ju|dx}\\
\\
&\leq&\ds{\Big|\int_\mathbb{K}(\p_x^n u)(1+\p_x)\Lambda^{-2}u^2dx\Big|+\sum_{k=0}^{n-1}\sum_{j=0}^k\binom{k}{j}k_{n-1}(t)I_n(t)}\\
\\
&\leq&\ds{\Big|\int_\mathbb{K}(\p_x^n u)(1+\p_x)\Lambda^{-2}u^2dx\Big|+(2^n-1)k_{n-1}(t)I_n(t)}.
\ea
$$
\end{proof}

\begin{proposition}\label{prop3.5}
    Recalling (\ref{1.3}), the following inequality holds 
    $$
    \Big|\int_\mathbb{K}u(1+\p_x)\Lambda^{-2}u^2dx\Big|\leq 2\|g\|_{L^\infty(\mathbb{K})}\|u\|_{L^1(\mathbb{K})}\|u\|_{L^2(\mathbb{K})}^2.
    $$
    
\end{proposition}

\begin{proof}
    We first note that $(1+\p_x)\Lambda^{-2}u^2=(g+\p_xg)\ast u^2$. Since $|\p_xg|\leq |g|$, by Young inequality we have 
    \bb\label{2.5}
    |(1+\p_x)\Lambda^{-2}u^2|\leq 2\|g\|_{L^\infty(\mathbb{K})} \|u^2\|_{L^1(\mathbb{K})}=2\|g\|_{L^\infty(\mathbb{K})} \|u\|^2_{L^2(\mathbb{K})}.
    \ee
    Taking this into account, we have
    $$
    \Big|\int_\mathbb{K}u(1+\p_x)\Lambda^{-2}u^2dx\Big|\leq \|(1+\p_x)\Lambda^{-2}u^2\|_{L^\infty(\mathbb{K})}\int_\mathbb{K}|u|dx,
    $$
    that implies the result.
\end{proof}

\begin{proposition}\label{prop3.6}
    If $u$ is a solution of \eqref{1.2} then
    $$
    \p_t\Big(\f{u^2}{2}\Big)=u\p_xu^2-u^3+u(1+\p_x)\Lambda^{-2}u^2,
    $$
    $$
    \p_t\Big(\f{u_x^2}{2}\Big)=u\p_xu_x^2+2u_x^3-u^2u_x-2uu_x^2+(\p_xu)(1+\p_x)\Lambda^{-2}u^2,
    $$
    and, for $n\geq 2$,
    $$
    \ba{lcl}
    \ds{\p_t \Big(\f{(\p_x^n u)^2}{2}\Big)}&=&\ds{u\p_x(\p_x^n u)^2+(n+1)(\p_x u)(\p_x^n u)^2+\sum_{k=0}^{n-2}\binom{n}{k}(\p_x^{n-k-1} u)(\p_x^k u)(\p_x^n u)}\\
    \\
    &&\ds{-\sum_{k=0}^{n}\sum_{j=0}^k\binom{k}{j}(\p_x^{k-j}u)(\p_x^j u)(\p_x^nu)(\p_x^nu)+(\p_x^n u)(1+\p_x)\Lambda^{-2}u^2.}
    \ea
    $$
\end{proposition}

\begin{proof}
    Let us first assume $n\geq 2$. Applying the operator $\p_x^n$ to \eqref{1.2}, we have
    $$
    \p_t(\p_x^nu)-\sum_{k=0}^{n+1}\binom{n+1}{k}(\p_x^{n+1-k}u)(\p_x^ku)=(1+\p_x)\Lambda^{-2}u^2-\sum_{k=0}^n\sum_{j=0}^k\binom{k}{j}(\p_x^{k-j}u)(\p_x^ju),
    $$
    where we used \eqref{2.1} and Proposition \ref{prop3.3}. After multiplying the equation above by $(\p_x^n u)$, rearranging the terms, taking into account the identity
    $$
    \ds{\sum_{k=0}^{n+1}\binom{n+1}{k}(\p_x^{n+1-k}u)(\p_x^ku)(\p_x^nu)}=\ds{u\p_x(\p_x^nu)^2+(n+1)(\p_xu)(\p_x^nu)^2}
    \ds{+\sum_{k=0}^{n-2}\binom{n}{k+1}(\p_x^{n-k}u)(\p_x^{k+1}u)\p_x^nu}
    $$
    we get the result.
    
    Cases $n=0$ and $n=1$ are straightforward and for this reason they are omitted.
\end{proof}

\begin{proposition}\label{prop3.7}
Let 
$$
\ba{lcl}\kappa_0(t)&=&2(k_0(t)+2\|g\|_{L^\infty(\mathbb{K})}\|u(t,\cdot)\|_{L^1(\mathbb{K})}),\\
\\
\kappa_m(t)&=&2[(n+3)k_1(t)+2\|g\|_{L^\infty(\mathbb{K})}\|u(t,\cdot)\|_{L^1(\mathbb{K})}+\al_m k_{m-1}(t)],
\ea
$$ 
where
$$
\al_m=\left\{\ba{lcl}
2^m+2^{m-2}+n,&\text{if}&m\geq 2,\\
\\
0, &\text{otherwise.}&
\ea\right.
$$
    If $u$ is a solution of \eqref{1.2}, then
    $$
    \f{d}{dt}\|\p_x^nu(t,\cdot)\|_{L^2(\mathbb{K})}^2\leq 2 \kappa_m(t)I_m,
    $$
for any $0\leq m\leq n$.
\end{proposition}

\begin{proof}
    Integrating the expressions in Proposition \ref{prop3.6} with respect to $x$ over $\mathbb{K}$, we get
    $$
    \f{d}{dt}\int_\mathbb{K} \f{u^2}{2}dx=\int_\mathbb{K}\Big(-u^3+u(1+\p_x)u^2\Big)dx \leq\Big(\|u\|_{L^\infty(\mathbb{K})}+2\|g\|_{L^\infty(\mathbb{K})}\|u\|_{L^1(\mathbb{K})}\Big)\|u\|_{L^2(\mathbb{K})}^2\leq \kappa_0(t)I_0(t).
    $$
    The procedure for the case $n=1$ is similar and, therefore, it is omitted. Let us tackle the arbitrary case $n\geq 2$.

    Integrating again with respect to $x$ over $X$, integrating by parts the first term on the right hand side, we get
    
    $$
    \ba{lcl}
    \ds{\f{d}{dt}\int_\mathbb{K} \f{(\p_xu)^2}{2}dx}&=&\ds{\int_\mathbb{K} u\p_x(\p_xu)^2dx+(n+1)\int_\mathbb{K} (\p_x u)(\p_x^n u)^2dx}\\
    \\
    &&\ds{+\sum_{k=0}^{n-2}\binom{n}{k}\int_\mathbb{K} (\p_x^{n-k-1} u)(\p_x^k u)(\p_x^n u)dx}\\
    \\
    &&\ds{-\sum_{k=0}^{n}\sum_{j=0}^k\binom{k}{j}\int_\mathbb{K} (\p_x^{k-j}u)(\p_x^j u)(\p_x^nu)(\p_x^nu)dx+\int_\mathbb{K} (\p_x^n u)(1+\p_x)\Lambda^{-2}u^2dx.}
    \ea
    $$

    Therefore, we have,
    $$
    \ba{lcl}
    \ds{\f{d}{dt}\f{\|\p_x^nu\|_{L^2(\mathbb{K})}^2}{2}}&\leq&\ds{ (n+2)\|u_x(t,\cdot)\|_{L^\infty(\mathbb{K})}\|\p_x^nu\|_{L^2(\mathbb{K})}^2+2\al_nk_{n-1}(t)I_n(t)}\\
    \\
    &&\ds{+2\|g\|_{L^\infty(\mathbb{K})}\|u(t,\cdot)\|_{L^1(\mathbb{K})}\|u(t,\cdot)\|_{L^2(\mathbb{K})}^2+2\|u(t,\cdot)\|_{L^\infty(\mathbb{K})}I_n(t)}\\
    \\
    &\leq&2[(n+3)k_1(t)+2\|g\|_{L^\infty(\mathbb{K})}\|u(t,\cdot)\|_{L^1(\mathbb{K})}+\al_m k_{m-1}(t)]I_n(t).
    \ea
    $$
    Above we used propositions \ref{prop3.4}--\ref{prop3.5} and the fact that
    $$
    \al_n=n+1+\sum\limits_{k=0}^{n-2}\binom{n}{k}+\sum\limits_{k=0}^{n}\sum_{j=0}^k\binom{k}{j}.
    $$

    From \eqref{2.2}--\eqref{2.4} and Proposition \ref{prop3.1}, we conclude that the right hand side of the last equation is bounded from above by $\kappa_n(t)I_n(t)$, for some positive and continuous function $\kappa_n$.
\end{proof}

\begin{remark}\label{rem3.1}
    From the proof of Proposition \ref{prop3.7} we conclude that $\kappa_n$ is determined by \eqref{2.2}. As a result, as long as the norms $\|\p_x^\ell u(t,\cdot)\|_{L^\infty(\mathbb{K})}$, $1\leq \ell\leq n-1$, are bounded, so is $\kappa_n(t)$.
\end{remark}

\begin{theorem}\label{teo2.1}
    Suppose that $u\in C^0([0,T);H^{n+2}(\mathbb{K}))\cap C^1([0,T);H^{n+1}(\mathbb{K}))$ is a solution to \eqref{1.1}. If, for some $m\in\{1,\cdots,n\}$ we have
    \bb\label{2.6}
    \int_0^t \kappa_m(\tau)d\tau<\infty,\quad t\in\R,
    \ee
    then $I_N(t)<\infty$ for any $t\in\R$ and $N\in[0,\cdots, n]$.
\end{theorem}

\begin{proof} 
We first observe the following: if, for some $m\leq n$ we have $I_m(t)<\infty$, for any $t\in\R$, then $\|u(t,\cdot)\|_{H^m}<\infty$, due to the relation  
$\|u\|_{H^m}=\sqrt{2I_m}.$

The Sobolev Embedding Theorem then implies that $\|\p_x^\ell u(t,\cdot)\|_{L^\infty(\mathbb{K})}<\infty$, for any $\ell\in\{1,\cdots,m-1\}$. In particular, $\|u(t,\cdot)\|_{L^\infty(\mathbb{K})}<\infty$, and so do $k_0(t)$, for any $t\in\R$.

Suppose that \eqref{2.6} holds for $m=n$.
    From Proposition \ref{prop3.7} we see that $k_\ell(t)\leq k_m(t)$, for any $\ell\leq m$. Therefore, we have
    $$
    \f{d}{dt}I_m(t)=\f{d}{dt}\sum_{k=0}^m\f{\|\p_x^ku(t,\cdot)\|_{L^2(\mathbb{K})}^2}{2}\leq m\kappa_m(t)I_m(t),
    $$
    and then, as long as \eqref{2.6} is valid for $m=n$, we have
    \bb\label{2.7}
    I_m(t)\leq I_m(0)e^{m\int_0^t \kappa_m(\tau)d\tau}<\infty.
    \ee

    Let us now assume that \eqref{2.6} holds for some $0\leq m<n$. Since 
    $$\kappa_m(t)=2[(m+3)k_1(t)+2\|g\|_{L^\infty(\mathbb{K})}\|u(t,\cdot)\|_{L^1(\mathbb{K})}+\al_m k_{m-1}(t)],$$
    the condition above tells us that
    \bb\label{2.8}
    \int_0^t \|\p_x^ju(\tau,\cdot)\|_{L^\infty(\mathbb{K})}d\tau<\infty,\quad 0\leq j\leq m-1
    \ee
    as well as
    \bb\label{2.9}
    \int_0^t\|u(\tau,\cdot)\|_{L^1(\mathbb{K})}d\tau<\infty,
    \ee
    for any $t\in\R$. However, \eqref{2.7} and \eqref{2.9} together implies that $k_{m+1}(t)$ is finite for any $t\in\R$. Repeating the argument, we reach to $m=n$. 
\end{proof}

\section{Global solutions}\label{sec3}

In view of Theorem \ref{teo2.1}, all we need to do is proving that $I_m(t)$ is bounded for finite values of $t$, for some $m\leq n$.

For both periodic and non-periodic cases, we observe that
$
\int_{\mathbb{K}}u_{xx}dx=0.
$
For the periodic case, this fact is enough to guarantee the existence of a point $\xi_t-1\in(0,1)$ such that $u_x(t,\xi_t-1)=0$, and $x_t$ such that $u_x(t,x_t)=0$ for each $t\in(0,T)$, respectively. 

Recalling the idea used in \cite{nilay}, we give the following lemma proved in a unified way.

\begin{lemma}\label{lem4.1}
Let $n\in \mathbb{N}. $ If $u_0\in H^n(\mathbb{K})\cap L^1(\mathbb{K})$, is such that $m_0\geq 0$, then there exists a constant $K>0$ such that the solution of \eqref{1.2} satisfies $\|u_x\|_{L^\infty(\mathbb{K})}\leq K$.
\end{lemma}
\begin{proof}
We first prove that $\|m(\cdot,t)\|_{L^1(\mathbb{K})}$ is constant for any $t$ as long as the solution exists. 
Recall (\ref{1.1}) and note that
$$
m_t=\p_t(u-u_{xx})=\p_x\Big((2-\p_x)(1+\p_x)u^2\Big)=\p_x\Big((1-\p_x^2)u^2+u^2)\Big).
$$
Integrating the relation above with respect to $x$ on $\mathbb{K}$, we obtain
$$
\f{d}{dt}\int_{\mathbb{K}} (u-u_{xx})dx=\Big((1-\p_x^2)u^2+u^2)\Big)\big|_{\mathbb{K}}=0,
$$
meaning that the $\|m(\cdot,t)\|_{L^1(\mathbb{K})}=\|m_0\|_{L^1(\mathbb{K})}.$

Now, assume that $m_0$ does not change sign and $m_0\geq 0$. Then,
$$
K:=\|m_0\|_{L^1(\mathbb{K})}=\|m(\cdot,t)\|_{L^1(\mathbb{K})}
$$
is a constant.

At this point, we pay attention to the subcases and write:
$$
K=\int_{\mathbb{K}} m(t,x)dr=\int_a^b m(t,x)dr
$$
is a constant, where $a$ and $b$ are constants determined by $\mathbb{K}$. Considering the periodic case ($\mathbb{K}=\s$), we get:
$$
\int_{\xi_{t}-1}^{\xi_t}m(t,x)dr \\
\geq\int_{\xi_{t}-1}^x (u-u_{xx})(r,t)dr=\int_{\xi_{t}-1}^x u(r,t)dr-u_x(x,t)\geq -u_x(x,t)$$

which holds for every $x\in[\xi_{t}-1, \xi_t]$. Similarly, the non-periodic case brings:
$$
\int_{-\infty}^{\infty}m(t,x)dr \\
\geq\int_{-\infty}^x (u-u_{xx})(r,t)dr=\int_{-\infty}^x u(r,t)dr-u_x(x,t)\geq -u_x(x,t)$$
which holds for every $x\in \R$. The last inequality, assuming that the improper integral is finite, is valid since 
$$
\frac{d}{dt}\int_{\R} u dx=\int_{\R} u_t dx=\int_{\R} [2uu_x+\partial_x\Lambda^{-2}(u^2+(u^2)_x]dx=0
$$
and $u_0\in L^1(\R)$.

For both periodic and non-periodic cases, we use the conclusions provided by \cite[Theorem 4.1]{nilay} and \cite[Lemma 5.5]{li-na} which guarantees that $u$ does not change sign provided that $m_0$ does not change sign. Taking into account the final results, we observe that $u_x$ is bounded from below in any case, since $K\geq u_x$.
%Moreover, 
%\begin{eqnarray*}
%    K=\int_{\xi_{t}-1}^{\xi_t}m(t,x)dr
%            \geq\int_{x}^{\xi_{t}} m(t,x)dr=\int_{x}^{\xi_{t}} u(r,t)dr+u_x(x,t)\geq u_x(x,t).
%\end{eqnarray*}
Hence, proceeding similarly as in the previous lines, we can show that $u_x$ is bounded also from above and we can conclude that $\|u_x\|_{L^\infty(\mathbb{K})}$ norm is bounded, i.e. $\|u_x\|_{L^\infty(\mathbb{K})}\leq K$. 
\end{proof}

\section{Blow-up of solutions} \label{sec4}

In this section, we shall study the conditions for solutions of \eqref{1.1} blow-up. After applying $(1-\p_x)^{-1}$ on both sides of \eqref{1.1} we can alternatively write it as
\bb\label{4.1}
\p_t(u-u_{x})=2u_x(u-u_x)+2u(u_x-u_{xx}).
\ee

The above operation is not valid in general, but it is licit for functions in $H^s(\mathbb{K})$, see Appendix \ref{appendix}, where a deduction for the Green function \eqref{1.5} is presented.

\begin{proposition}\label{prop4.1}
    Assume that $u$ is a solution to \eqref{1.1} subject to an initial datum $u_0\in H^4(\mathbb{K})$. Fix $x\in\mathbb{K}$ and consider the problem
    \bb\label{4.2}
    \left\{
    \ba{lcl}
    \ds{\f{d}{dt}q(t,x)}&=&-2u(t,q(t,x)),\\
    \\
    q(0,x) &=&x.
    \ea
    \right.
    \ee

Then the problem has a unique solution $q\in C^1([0,T)\times\R)$. Moreover, 
\bb\label{4.3}
q_x(t,x)=e^{-2\int_0^t u_x(s,q(s,x))ds}>0
\ee
and the function $\phi(t,x)=(t,q(t,x))$ is a $C^0$ bijection fixing $\{0\}\times\R$, such that its restriction to any open set $X$ of $[0,T)\times\mathbb{K}$ is a diffeomorphism between $X$ and $\phi(X)$.
\end{proposition}

\begin{proof}
The existence and uniqueness of solutions to \eqref{4.2} can be proved following step-by-step the proof given in \cite[Theorem 3.1]{const2000-1} regardless $\mathbb{K}$ due to the local nature of the problem. In addition, the same proof shows also \eqref{4.3}. In particular, it implies that $\R\ni x\mapsto q(t,x)$ is a diffeomorfism of the line, for each fixed $t\in[0,T)$. This fact then implies that $\phi(t,x)=(t,q(t,x))$ is a bijection from $[0,T)\times\mathbb{K}$ into its image, $\phi(0,x)=(0,q(0,x))=(0,x)$. The fact that $\phi$ is a diffeomorphism between any open set $X$ of $[0,T)\times\mathbb{K}$ and its image follows from the fact that $q_x(t,\cdot)$ is a diffeomorphism for each fixed $t$.
\end{proof}

\begin{proposition}\label{prop4.2}
    Let $h(\cdot)$ and $w(\cdot,\cdot)$ be bounded and continuous functions, such that $h(x_0)=-\sigma_0$, $\sigma_0>0$, for some $x_0\in\mathbb{K}$, and $\inf\limits_{x\in\mathbb{K}}w(t,x)\geq -L$, for some $L>0$. For $t\geq0$ and $x\in\mathbb{K}$, define
    $$
    f(t,x)=\f{h(x)e^{\ds{{\int_0^t w(\tau,x)d\tau}}}}{1+2h(x)\ds{\int_0^t e^{\ds{\int_0^\tau w(s,x)ds}d\tau}}}.
    $$

If $2\sigma_0-L\in(0,1)$, then $f$ cannot be defined for all $t>0$. More precisely, we have
$$
f(t,x_0)\rightarrow-\infty\quad\text{as}\quad t\nnearrow T_0,
$$
for some $T_0=T_0(L,\sigma_0)$.
\end{proposition}

\begin{proof}
Under the given conditions, we have $e^{\int_0^t w(\tau,x)d\tau}\geq e^{-Lt}$ and then,
$$
\int_0^t e^{\ds{\int_0^\tau w(s,x)ds}}d\tau\geq \int_0^te^{-L\tau}d\tau=\f{1}{L}\Big(1-e^{-Lt}\Big).
$$

On the one hand, we have
$$
-2\sigma_0\int_0^t e^{\ds{\int_0^\tau w(s,x)ds}}d\tau\leq-\f{2\sigma_0}{L}(1-e^{-Lt}),
$$
that implies, at least for small values of $t$, 
$$
0<1-2\sigma_0\int_0^t e^{\ds{\int_0^\tau w(s,x)ds}}d\tau\leq 1-\f{2\sigma_0}{L}(1-e^{-Lt})
$$
and then,
\bb\label{4.4}
0<\f{L}{L-2\sigma_0(1-e^{-Lt})}\leq \f{1}{1-2\sigma_0\ds{\int_0^t e^{\ds{\int_0^\tau w(s,x)ds}}d\tau}}.
\ee

Combining the fact that
$$
-2\sigma_0e^{\ds{\int_0^t w(\tau,x)d\tau}}<0
$$
with \eqref{4.4} and the fact that $h(x_0)=-\sigma_0$, we have
\bb\label{4.5}
f(t,x_0)=\f{-2\sigma_0 e^{\ds{{\int_0^t w(\tau,x)d\tau}}}}{1-2\sigma_0\ds{\int_0^t e^{\ds{\int_0^\tau w(s,x)ds}d\tau}}}\leq\f{-2L\sigma_0}{L-2\sigma_0(1-e^{-Lt})}.
\ee

Let
$$
T_0:=-\f{1}{L}\ln{\Big(1-\f{L}{2\sigma_0}\Big)}.
$$

From \eqref{4.5} we conclude that $f(t,x_0)\rightarrow-\infty$ as $t\nnearrow T_0$.
\end{proof}

\begin{proposition}\label{prop4.3}
Let $u_0\in H^4(\mathbb{K})$. If $u_0-u_0'$ is either non-negative or non-positive, then the corresponding solution to \eqref{1.1} will inherit the same sign, in the following sense:
$$
\sign{(u-u_x)(\phi(t,x))}=\sign{(u-u_0')(x)},
$$
where $\phi(\cdot,\cdot)$ is the function given in Proposition \ref{prop4.1}.
\end{proposition}

\begin{proof}
    Let $q$ and $\phi$ as given by Proposition \ref{prop4.1}. From \eqref{4.1} and \eqref{4.2}, we have
\bb\label{4.6}
\p_t(u-u_x)(\phi(t,x))=(u_t-u_{tx}-2u(u_x-u_{xx}))(\phi(t,x))=2[u_x(u-u_x)](\phi(t,x)).
\ee

Integrating the above equation, we have
\bb\label{4.7}
(u-u_x)(\phi(t,x))=(u-u_0')(x)e^{2\ds{\int_0^t u_x(\phi(\tau,x))}d\tau}.
\ee

The result is a consequence of \eqref{4.7}.
\end{proof}

\begin{proposition}\label{prop4.4}
    Assume that $u_0\in H^4(\mathbb{K})$ is a non-trivial initial datum satisfying $(u_0-u_0')(x)\leq 0$, for all $x\in\mathbb{K}$, and $u_0-u_0'\in L^1$. Then $\|(u-u_x)(t,x)\|_{L^1(\mathbb{K})}$ is constant.
\end{proposition}

\begin{proof}
Let $u$ be the corresponding solution to \eqref{1.1}. Then it belongs to $C^0([0,T),H^4(\mathbb{K}))\cap C^1([0,T),H^3(\mathbb{K}))$ and thus, it satisfies \eqref{4.1}, that is equivalent to
    $$\p_t(u-u_{x})=\p_x[2u(u-u_{xx})].$$
    Integrating with respect to $x$ over $\mathbb{K}$, we have
    $$
    \f{d}{dt}\int_\mathbb{K}(u-u_{x})dx=[2u(u-u_{xx})]\Big|_\mathbb{K}=0,
    $$
    meaning that
    $$
    \int_\mathbb{K}(u-u_{x})(t,x)dx=\int_\mathbb{K}(u-u_0')(x)dx.
    $$

    Since $u_0$ is non-trivial, then $\|u_0-u_0'\|_{L^1(\mathbb{K})}>0$. On the other hand, by Proposition \ref{prop4.3} and the fact that $u_0-u_0'\leq0$, we have
    $$
    \|u-u_x\|_{L^1(\mathbb{K})}=\int_{\mathbb{K}}|(u-u_x)(t,x)|dx=-\int_\mathbb{K}(u-u_{x})(t,x)dx=-\int_\mathbb{K}(u-u_0')(x)dx,
    $$
    that proves the result.
\end{proof}

Before proving Theorem \ref{teo1.2}, we observe that from \eqref{1.5} we have
$$
\|G\|_{L^\infty(\mathbb{K})}=\left\{\ba{ll}
    \ds{1}, &\,\,\text{when }\mathbb{K}=\R,\\
    \\
    \ds{\f{e}{e-1}},&\,\,\text{when }\mathbb{K}=\s.
\ea
\right.
$$

{\bf Proof of Theorem \ref{teo1.2}}. Assume that $(u_0-u_0')(x_0)=-\sigma_0<0$, for some $x_0\in\mathbb{K}$. By Proposition \ref{prop4.3} we have $(u-u_x)(t,x)<0$ as long as the solution exists. Since
$$
(1-\p_x)u=u-u_x\Rightarrow u=(1-\p_x)^{-1}(u-u_{x})=G\ast (u-u_x),
$$
where $G$ is given by \eqref{1.5}, we have $u(t,x)<0$ whenever it is defined. Let 
$$\psi=(u-u_x)\circ\phi,$$
where $\phi$ is the function given in Proposition \ref{prop4.1}. Then
$$
u_x\circ\phi=u\circ\phi-\psi=:w-\psi.
$$

With these new functions, \eqref{4.6} is equivalent to
\bb\label{4.8}
\p_t\psi-2w\psi+2\psi^2=0.
\ee

Therefore, from \eqref{4.8}, \eqref{4.1} and \eqref{1.1}, we obtain the Cauchy problem
\bb\label{4.9}
\left\{
\ba{l}
\p_t\psi-2w\psi+2\psi^2=0,\\
\\
\psi(0,x)=(u_0-u_0')(x)=:h(x).
\ea
\right.
\ee

Since \eqref{4.8} is a Bernoulli equation, the solution of \eqref{4.9} is
\bb\label{4.10}
\psi(t,x)=\f{h(x)e^{\ds{{\int_0^t 2w(\tau,x)d\tau}}}}{1+2h(x)\ds{\int_0^t e^{\ds{\int_0^\tau 2w(s,x)ds}d\tau}}}.
\ee

We now observe that $w(t,x)=u(\phi(t,x))$. Therefore,
$$
\ba{lcl}
\inf\limits_{x\in\mathbb{K}}w(t,x)&=&\inf\limits_{x\in\mathbb{K}}(u\circ\phi)(t,x)\geq-\|u(t,\cdot)\|_{L^\infty(\mathbb{K})}=-\|G\ast(u-u_x)(t,\cdot)\|_{L^\infty(\mathbb{K})}\\
\\
&=&-\|G\ast(u_0-u_0')(t,\cdot)\|_{L^\infty(\mathbb{K})}\geq-\|G\|_{L^\infty(\mathbb{K})}\|u_0-u_0'\|_{L^1(\mathbb{K})}.
\ea
$$

We now observe that the conditions on Theorem \ref{teo1.2}'s statement say that
$$
\sigma_0-\|G\|_{L^\infty(\mathbb{K})}\|u_0-u_0'\|_{L^1(\mathbb{K})}\in(0,1/2)
$$
and if we replace $w(\cdot,\cdot)$ by $w(\cdot,\cdot)/2$, then \eqref{4.10} is nothing but the function $f$ given in Proposition \ref{prop4.2}. The result follows from that proposition replacing $L$ by $2\|G\|_{L^\infty(\mathbb{K})}\|u_0-u_0'\|_{L^1(\mathbb{K})}$. \hfill$\square$

\section{Proof of the geometric results}\label{sec5}

In this section we prove theorems \ref{teo1.3} and \ref{1.4}. We begin with the following observation: as long as $u_0\in H^4(\mathbb{K})$, Remark \ref{rem1.2} says that we then have a unique local solution $u\in C^0([0,T), H^s(\mathbb{K}))\cap C^1([0,T), H^{s-1}(\mathbb{K}))$. As a result, the forms \eqref{1.6} are defined on
\bb\label{5.1}
U_T=(0,T)\times\R.
\ee

Moreover, the proof that \eqref{1.6} defines a PSS follows from \cite[Theorem 1]{freire-tito-sam} choosing $m_1=1$ in \cite[Equation (8)]{freire-tito-sam} and for this reason it is omitted.

\subsection{Proof of Theorem \ref{teo1.3}}

As long as $u_0-u_0''>0$, by Theorem \ref{teo1.1} $u$ is global, meaning that it can be defined on $U=[0,\infty)\times\R$ and additionally satisfies $u(t,x+1)=u(t,x)$ for the periodic case. As a result, the one forms \eqref{1.6} are defined on $U_\infty=(0,\infty)\times\R$.

The conditions on the initial datum imply that $u>0$ and it cannot be constant on $U_\infty$, meaning that for open sets $\Omega\subseteq U$, we have $\nabla u(p)\neq0$, $p\in\Omega$. Without loss of generality, we may assume that $\Omega$ is a connected component of the set $\{p\in U,\,\,\nabla u(p)\neq0\}$. By \cite[Remark 1]{freire-tito-sam}, $\Omega$ is ensured with a PSS structure in the sense discussed in \cite[Section II]{reyes2000jmp}

\subsection{Proof of Theorem \ref{teo1.4}}

It suffices proving that under the conditions on the initial datum, then
\bb\label{5.2}
g_{11}=(u-u_{xx})^2+\Big(\mu (u-u_{xx})\pm \sqrt{1+\mu^2}\Big)^2\geq (u-u_{xx})^2
\ee
and
\bb\label{5.3}
g_{22}=(1+\mu^2)\Big(2u(u-u_{xx})-2(u-u_x)^2\Big)^2
\ee
cannot be simultaneously bounded.

Let $u_0$ and $x_0$ as in Theorem \ref{teo1.2}, $(t_n,x_n)$ be a sequence converging to $(T_0,x_0)$, such that $t_n\in(0,T_0)$. As such, $(t_nx_n)\in U_{T_0}$ (see \eqref{5.1}). Define $g_1=g_{11}\circ\phi$ and $g_2=g_{22}\circ\phi$, where $\phi$ is given by Proposition \ref{prop4.1}. The proof of Theorem \ref{teo1.2} tells us that 
$$(u-u_x)\circ\phi(t_n,x_n)\rightarrow-\infty\quad\text{as}\quad n\rightarrow\infty.$$ 
Therefore, if \eqref{5.3} is bounded, then 
$$u(u-u_{xx})(t_n,x_n)=O(((u-u_x)\circ\phi)(t_n,x_n)^2)\quad\text{as}\quad n\rightarrow\infty,$$ 
meaning that \eqref{5.2} cannot be bounded.

We now suppose that \eqref{5.2} is bounded. Therefore, for $n\gg 1$, we have
$$g_2(t_n,x_n)\approx 4(1+\mu^2)(u-u_x)(t_n,x_n)^2,$$
that cannot be bounded.

Finally, we now note that if either $g_{11}$ or $g_{22}$ is not bounded, then neither is $g_{12}$.

\section*{Acknowledgements} N. D. Mutluba\c{s} is supported by the Turkish Academy of Sciences within the framework of the Outstanding Young Scientists Awards Program (T\"{U}BA-GEBIP-2022) and Tubitak 1001 project (grant number T.A.CF-24-02925). I. L. Freire is thankful to CNPQ (grant number 310074/2021-5) for financial support. The authors would like to thank FAPESP (grant number 2024/01437-8) for financial support.

{\bf Declarations of interest:} none.

{\bf Declarations of AI:} AI has not been used in this work.

\appendix

\section{The Green function of the operator $1-\p_x$}\label{appendix}

Let us derive the Green function \eqref{1.5}. We begin with the periodic case.

The solution of the equation $y'(x)=y(x)$ is $y(x)=Ce^{x}$, for some $C\in\R$. If we impose that $y$ is periodic, with period $1$, we can measure its jump $[y_0]$ at $x=0$. In fact, since $y$ is 1-periodic, then $y(x)=y(x+1)$, for $x\in(-1,0)$. Then,
$$
[y_0]=\lim_{\epsilon \ssearrow 0}(y(\epsilon)-y(-\epsilon)=\lim_{\epsilon \ssearrow 0}(y(\epsilon)-y(1-\epsilon))=\lim_{\epsilon \ssearrow 0}(Ce^\epsilon-Ce^{1-\epsilon})=C(1-e).
$$

The Green function $G$ can be found by imposing that $G$ is a distribution that agrees with $y$ but has a jump in its first derivative at $x=0$, that is, $D_xG=y'+[y_0]\delta(x)$, where $'$ denotes classical derivative. Then
$$
(1-D_x)G=y-(y'+[y_0]\delta(x))=-[y_0]\delta(x).
$$

On the other hand, the Green function should satisfy $(1-D_x)G=\delta(x)$. Comparing this with the last equation, we conclude that $[y_0]=-1$, that is,
$$
C=\f{1}{e-1}.
$$

As a result, for $x\in(0,1)$ we have
$$G(x)=\f{e^x}{e-1}.$$

We can extend it periodically to the whole line by taking
$$
G(x)=\f{e^{x-\lfloor x \rfloor}}{e-1},\quad x\in\R.
$$

We now observe that if $v=(1-\p_x)u$, we can formally find its periodic Fourier transform $\hat{v}(k)=(1-2\pi i k)\hat{u}(k)$ and
$$
\|v\|_{H^s(\s)}=\sum_{k\in\mathbb{Z}}(1+4\pi^2k^2)^s|\hat{v}(k)|^2=\sum_{k\in\mathbb{Z}}(1+4\pi^2k^2)^s|1-2\pi i k|^2|\hat{u}(k)|^2=\sum_{k\in\mathbb{Z}}(1+4\pi^2k^2)^{s+1}|\hat{u}(k)|^2=\|u\|^2_{H^{s+1}(\s)},
$$
meaning that for the spaces we are working on this paper, $(1-\p_x)$ is an isomorphism and the solutions of \eqref{1.1}, \eqref{1.2} and \eqref{4.1} are the same.

For the non-periodic case, the Green function can be found by taking the Fourier transform of the equation
$$
(1-\p_x)G=\delta,
$$
and then use the inverse Fourier transform to get \eqref{1.5}.


\begin{thebibliography}{0}

\bibitem{chprl} R. Camassa and D. D. Holm, An integrable shallow water equation with peaked solitons, Physical Reviews Letters, vol. 71, 1661--1664, (1993).

\bibitem{chern} S. S. Chern and K. Tenenblat, Pseudospherical surfaces and evolution equations, Stud. Appl. Math., vol. 74, 55--83, (1986).

\bibitem{const1997} A. Constantin, On the Cauchy problem for the periodic Camassa-Holm equation, J. Differ. Equations, vol. 141, 218-235, (1997).

\bibitem{const1998-1} A. Constantin, J. Escher, Global existence and blow-up for a shallow water equation, Annali Sc. Norm. Sup. Pisa, vol. 26, 303--328, (1998).

\bibitem{const1998-2} A. Constantin and J. Escher, Wave breaking for nonlinear nonlocal shallow water equations, Acta Math., vol. 181, 229--243 (1998).

\bibitem{const1998-3} A. Constantin and J. Escher, Well-posedness, global existence, and blowup phenomena for a periodic quasi-linear hyperbolic equation, Commun. Pure Appl. Math., vol. 51, 475-504, (1998).

\bibitem{const2000-1} A. Constantin, Existence of permanent and breaking waves for a shallow water equation: a geometric approach, Ann. Inst. Fourier, vol. 50, 321--362, (2000).

\bibitem{nilay} N. Duruk Mutlubas and I. L. Freire, Existence and uniqueness of periodic pseudospherical surfaces emanating from Cauchy problems, Proc. R. Soc. A., vol. 480, paper 20230670, (2024).


\bibitem{li-na} J. Li and Z. Yin, Well-poseness and global existence for a generalized Degasperis--Procesi equation, Nonlinear Anal RWA, vol. 28, 72--92, (2016).

\bibitem{li-jmaa} M. Li and Z. Yin, Global solutions and blow-up phenomena for a generalized Degasperis–Procesi equation, J. Math. Anal. Appl., vol. 478, 604--624, (2019).

\bibitem{liu-jde}Y. Mi, Y. Liu, B. Guo, and T. Luo, The Cauchy problem for a generalized Camassa-Holm equation, J. Diff. Equ., vol. 266, 6739--6770, (2019).

\bibitem{pri-na} P. L. da Silva, Global analytic solutions of a pseudospherical Novikov equation, Nonlin. Anal. TMA, vol. 251, paper 113689, (2025).

\bibitem{pri-jde} P. L. da Silva, I. L. Freire and N. Sales Filho, An integrable pseudospherical equation with pseudo-peakon solutions, J. Diff. Equ., vol. 419, 291-323, (2025).


\bibitem{freire-tito-sam} I. L. Freire and R. S. Tito, A Novikov equation describing pseudospherical surfaces, its pseudo-potentials, and local isometric immersions, Studies Appl. Math., vol. 148, 758--772, (2022).

\bibitem{freire-AML} I. L. Freire, Remarks on strong global solutions of the $b-$equation, Appl. Math. Lett., vol. 146, paper 108820, (2023).


\bibitem{freire-ch} I. L. Freire,  Breakdown of pseudospherical surfaces determined by the Camassa-Holm equation, J. Diff. Equ., vol. 378, 339--359, (2023).

\bibitem{freire-dp} I. L. Freire, Local isometric immersions and breakdown of manifolds determined by Cauchy problems involving the Degasperis-Procesi equation, J Nonlinear Sci 35, paper 3 (2025).

\bibitem{freire2026} I. L. Freire,  A look on equations describing pseudospherical surfaces, arXiv:2506.23890, (2025).

\bibitem{kasman} A. Kasman, Glimpses of soliton theory: the algebra and geometry of nonlinear PDEs (2nd Edition), American Mathematical Society (2023).

\bibitem{guo} Y. Wang and Y. Guo, Blow-up criterion and persistence property to a generalized Camassa–Holm equation, Symmetry, vol. 15, 493, (2023).


\bibitem{nov} V. Novikov, Generalizations of the Camassa–Holm equation, J. Phys. A: Math. Theor., vol. 42, paper 342002, (2009).

\bibitem{reyesjmp2000} E. G. Reyes, Conservation laws and Calapso-Guichard deformations of equations describing pseudo-spherical surfaces, J. Math. Phys., vol. 41, 2968-2979, (2000).

\bibitem{reyes2000} E. G. Reyes, Some geometric aspects of integrability of differential equations in two independent variables, Acta Appl. Math., vol. 64, 75--109, (2000).

\bibitem{reyes2002} E. G. Reyes, Geometric integrability of the Camassa-Holm equation, Lett. Math. Phys., vol. 59, 117--131, (2002).

\bibitem{reyes2006-sel} E. G. Reyes, Pseudo-potentials, nonlocal symmetries and integrability of some shallow water equations, Sel. Math., New Ser., vol. 12, 241--270, (2006).

\bibitem{reyes2006-jde} E. G. Reyes, Correspondence theorems for hierarchies of equations of pseudospherical type, J. Diff. Equ., vol 225, 26--56, (2006).


\bibitem{reyes2000jmp} E. G. Reyes, Conservation laws and Calapso-Guichard deformations of equations describing pseudospherical surfaces. J. Math. Phys., vol. 41, 2968-2989, (2000).



\bibitem{reyes2011} E. G. Reyes, Equations of pseudospherical type (After S. S. Chern and K. Tenenblat), Results. Math., vol 60, 53--101, (2011).

\bibitem{nazime} N. Sales Filho and I. L. Freire, Structural and qualitative properties of a geometrically integrable equation, Commun. Nonlin. Sci. Num. Simul., vol. 114, paper 106668, (2022). 



\bibitem{tito} R. S. Tito, Equações descrevendo superfícies pseudo-esféricas, MSc dissertation, Universidade Federal do ABC (2022). (in Portuguese)

\bibitem{zab} N. J. Zabusky and M. D. Kruskal, Interaction of ``solitons'' in a collisionless plasma and the recurrence of initial states, Phys. Rev. Lett., vol. 15, 240--243, (1965).


\end{thebibliography}
\end{document}